\documentclass[11pt]{amsart}
\usepackage{graphicx,amscd,color}

\theoremstyle{plain}
\newtheorem{theorem}{Theorem}[section]
\newtheorem{proposition}[theorem]{Proposition}
\newtheorem{lemma}[theorem]{Lemma}

\newtheorem{definition}[theorem]{Definition}
\newtheorem{question}[theorem]{Question}

\theoremstyle{remark}

\newtheorem*{claim}{Claim}

\title [On the connectedness of subcomplexes of a disk complex]
{On the connectedness of subcomplexes of a disk complex}

\author[J. H. Lee]{Jung Hoon Lee}
\address{Department of Mathematics and Institute of Pure and Applied Mathematics,
Chonbuk National University, Jeonju 561-756, Korea}
\email{junghoon@jbnu.ac.kr}

\begin{document}

\begin{abstract}
For a boundary-reducible $3$-manifold $M$ with $\partial M$ a genus $g$ surface,
we show that if $M$ admits a genus $g+1$ Heegaard surface $S$,
then the disk complex of $S$ is simply connected.
Also we consider the connectedness of the complex of reducing spheres.
We investigate the intersection of two reducing spheres
for a genus three Heegaard splitting of $\mathrm{(torus)} \times I$.
\end{abstract}

\maketitle

\section{Introduction}\label{sec1}

One of the ways to understand a $3$-manifold is to study nicely embedded surfaces in the $3$-manifold.
An {\em incompressible surface} is a surface which has no compressing disks, and
it has played an important role in the $3$-manifold topology.
Let $S$ be a closed orientable surface embedded in a $3$-manifold and
suppose that $S$ compresses to both sides of $S$.
If every compressing disk for $S$ in one side intersects every compressing disk for $S$ in the other side,
then the surface is called a {\em strongly irreducible surface}.

Let $V$ denote one side of $S$ and $W$ denote the other side of $S$, and
let $\mathcal{D}_V$ and $\mathcal{D}_W$ denote the set of all compressing disks in $V$ and $W$ respectively.
The surface $S$ is called a {\em critical surface} if
the set of all compressing disks for $S$ can be partitioned
into two non-empty sets $C_0$ and $C_1$ satisfying the following conditions.

\begin{enumerate}
\item For each $i=0,1$, there is at least one pair of compressing disks
$D_i \in \mathcal{D}_V \cap C_i$ and $E_i \in \mathcal{D}_W \cap C_i$ such that $D_i \cap E_i = \emptyset$.
\item If $D \in \mathcal{D}_V \cap C_i$ and $E \in \mathcal{D}_W \cap C_{1-i}$, then $D \cap E \ne \emptyset$.
\end{enumerate}

Incompressible surfaces, strongly irreducible surfaces and critical surfaces share nice properties.
For example, if an irreducible $3$-manifold contains an incompressible surface and one of these surfaces,
then the two surfaces can be isotoped so that any intersection loop is essential on both surfaces.
Extending these notions,
Bachman defined topologically minimal surfaces in terms of the disk complex of a surface \cite{Bachman}.

A {\em disk complex} $\mathcal{D}(S)$ of a surface $S$ embedded in a $3$-manifold
is a simplicial complex defined as follows.

\begin{itemize}
\item Vertices of $\mathcal{D}(S)$ are isotopy classes of compressing disks for $S$.
\item A collection of $k+1$ distinct vertices forms a $k$-simplex if there are pairwise disjoint representatives.
\end{itemize}

A surface $S$ is {\em topologically minimal} if
$\mathcal{D}(S)$ is empty or $\pi_i(\mathcal{D}(S))$ is non-trivial for some $i$.
A {\em topological index} of $S$ is $0$ if $\mathcal{D}(S)$ is empty, and
the smallest $n$ such that $\pi_{n-1}(\mathcal{D}(S))$ is non-trivial, otherwise.
The topological indices of an incompressible surface, a strongly irreducible surface and
a critical surface are $0,1,2$ respectively.
There exist topologically minimal surfaces of high index, e.g. \cite{Bachman-Johnson}, \cite{Lee}.

In general, it is not easy to understand topologically minimal surfaces of high index.
A motivation of this paper is to understand
the structure of a disk complex of a topologically minimal surface of high index$(\ge 2)$.
As a first step we study about $2$-dimensional simplices of a disk complex,
i.e. a triple of mutually disjoint compressing disks.
The situation we consider in Section \ref{sec2} is that
there is still a weak reducing pair after we compress a Heegaard surface once.
We show that for a genus $g \ge 3$ Heegaard splitting of $S^3$ and a given (generic) weak reducing pair $(D,E)$,
there exists a weak reducing triple $(D,D',E)$ or $(D,E,E')$, say $(D,D',E)$, such that
$(D',E)$ is a weak reducing pair after compressing along $D$ (Theorem \ref{thm2.3}).

In Section \ref{sec3}, we consider the case of two weak reducing pairs $(D_1,E_1)$ and $(D_2,E_2)$ such that
adding two disjoint $2$-handles $N(D_1)$ and $N(D_2)$ results a boundary-reducible manifold.
We apply this observation to a boundary-reducible $3$-manifold $M$.
We show that if $\partial M$ is a genus $g$ surface and $M$ admits a genus $g+1$ Heegaard surface $S$,
then the disk complex of $S$ is simply connected (Theorem \ref{thm3.3}).
In Section \ref{sec4}, we consider a question on the connectedness of the complex of reducing spheres.
We investigate the intersection of two reducing spheres
for a genus three Heegaard splitting of $\mathrm{(torus)} \times I$ (Proposition \ref{prop4.5}).

\section{A weak reducing pair after a compression}\label{sec2}

For a closed $3$-manifold $M$, a {\em Heegaard splitting} $V \cup_S W$ is
a decomposition of $M$ into two handlebodies $V$ and $W$,
where $M = V \cup W$ and $V \cap W = \partial V = \partial W = S$.

A Heegaard splitting can be defined also for a $3$-manifold with non-empty boundary.
Let $S$ be a closed orientable surface.
A {\em compression body} $V$ is a connected $3$-manifold obtained from $S \times I$
by attaching $2$-handles to $S \times \{1\}$
and capping of any resulting $2$-sphere boundary components with $3$-balls.
The surface $S \times \{0\}$ is denoted by $\partial_+ V$ and $\partial V - \partial_+ V$ is denoted by $\partial_- V$.
If $\partial_- V = \emptyset$, then it is a handlebody.
A collection of essential disks of $V$ is called a {\em complete essential disk system}
if cutting $V$ along the disks results $\partial_-V \times I$ when $\partial_- V \ne \emptyset$
and a $3$-ball when $\partial_- V = \emptyset$.
For a $3$-manifold $M$, a {\em Heegaard splitting} $V \cup_S W$
is a decomposition of $M$ into two compression bodies $V$ and $W$,
where $M = V \cup W$ and $V \cap W = \partial_+ V = \partial_+ W = S$ and $\partial M = \partial_- V \cup \partial_- W$.

Sometimes even though it is not a Heegaard splitting, we will use the notation $V \cup_S W$ if $V \cap W = S$.

\begin{definition}\label{def1}
For $V \cup_S W$, two essential disks $D \subset V$ and $E \subset W$
are called a {\em weak reducing pair}, denoted by $(D,E)$, if $D \cap E = \emptyset$.
A weak reducing pair $(D,E)$ is {\em generic} if $D$ is essential in $V \cup N(E)$ and $E$ is essential in $W \cup N(D)$.
\end{definition}

\begin{definition}\label{def2}
For $V \cup_S W$, three essential disks $D_1, D_2 \subset V$ and $E \subset W$
are called a {\em weak reducing triple}, denoted by $(D_1,D_2,E)$, if $D_1 \cap D_2 = D_1 \cap E = D_2 \cap E = \emptyset$.
Similarly, three essential disks $D \subset V$ and $E_1,E_2 \subset W$
are called a {\em weak reducing triple}, denoted by $(D,E_1,E_2)$, if $D \cap E_1 = D \cap E_2 = E_1 \cap E_2 = \emptyset$.
\end{definition}

Let $V \cup_S W$ be a Heegaard splitting of $S^3$ of genus $g \ge 3$.
Let $(D,E)$ be a weak reducing pair.
It is easy to find a weak reducing triple containing $(D,E)$.
If $D$ is separating in $V$, we can take a nonseparating disk $D'$ in a component of $V-D$
such that $(D,D',E)$ is a weak reducing triple.
So we may assume that $D$ is a nonseparating disk.
Take two parallel copies of $D$ and attach a band along an arc in $\partial V - \partial E$ to get a disk $D'$ in $V$.
Then $(D,D',E)$ is a weak reducing triple.
But such a disk $D'$ becomes inessential if we compress $V$ along $D$.

So what is interesting is a weak reducing triple $(D,D',E)$ such that
$(D',E)$ is a weak reducing pair after we compress $V$ along $D$,
or equivalently after we add a $2$-handle $N(D)$ to $W$.
Here we assume that $(D,E)$ is generic to ensure that $E$ is essential after the compression.
Let $V'$ be the handlebody obtained by compressing $V$ along $D$.
However, it would not be possible to find a weak reducing pair $(D',E)$ if
$\partial E$ is a disk-busting curve in $\partial V'$. See Figure \ref{fig1}.

\begin{figure}[ht!]
\begin{center}
\includegraphics[width=7cm]{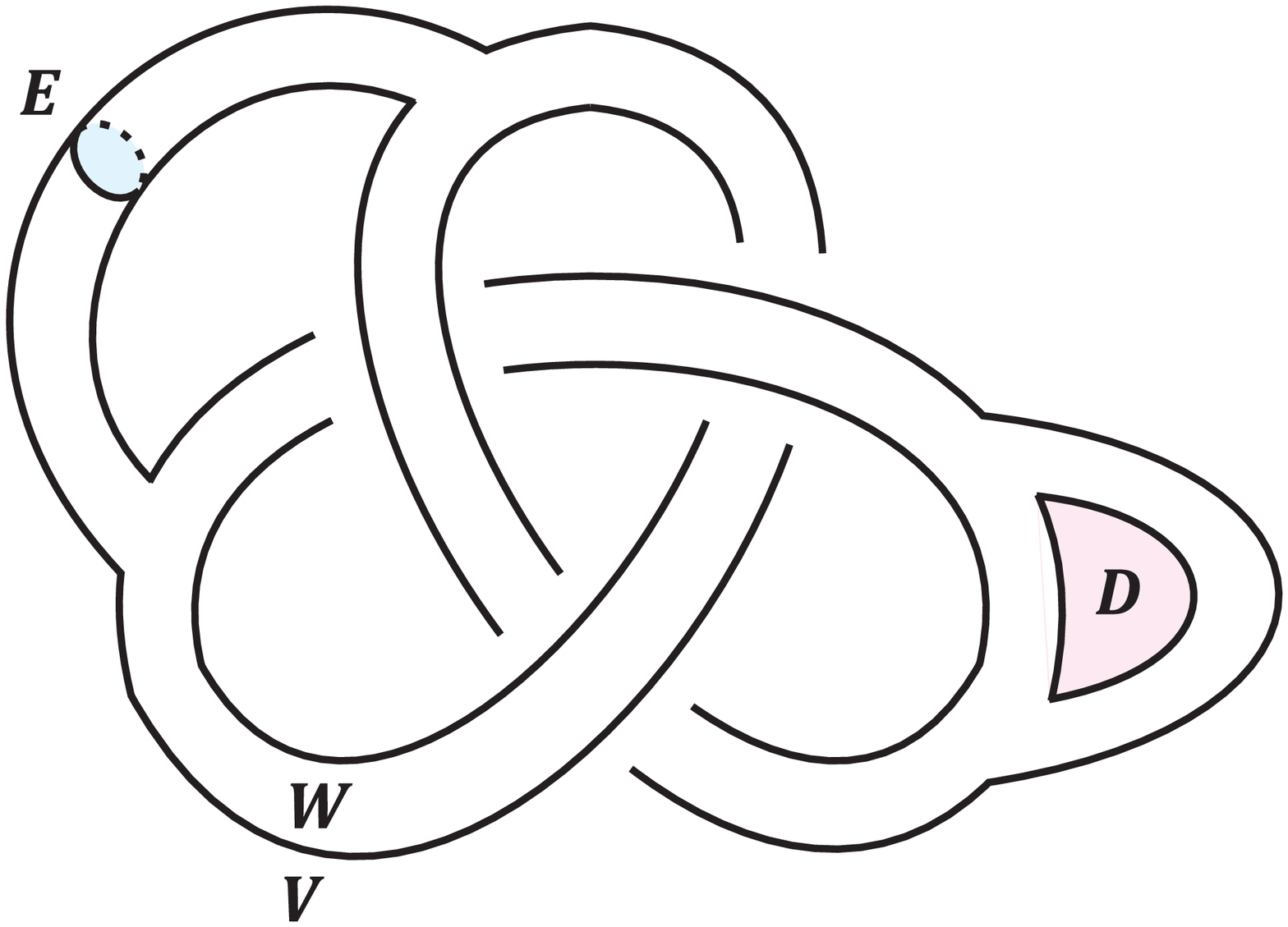}
\caption{}\label{fig1}
\end{center}
\end{figure}

Nevertheless in that case, we can find a weak reducing triple $(D,E,E')$ such that
$(D,E')$ is a weak reducing pair after compressing $W$ along $E$.

\begin{theorem}\label{thm2.3}
Suppose that $(D,E)$ is a generic weak reducing pair for a genus $g \ge 3$ Heegaard splitting $V \cup _S W$ of $S^3$.
Then there exists either a weak reducing triple $(D,D',E)$ or $(D,E,E')$, say $(D,D',E)$, such that
$(D',E)$ is a weak reducing pair after compressing $V$ along $D$.
\end{theorem}

\begin{proof}
Compress $S$ along $D \cup E$ and let $S'$ be the resulting surface.
Since every closed surface of positive genus in $S^3$ is compressible, we can find a compressing disk for a component of $S'$.
If the compressing disk intersects other components (if any),
we can find another compressing disk whose interior is disjoint from $S'$ by standard innermost disk argument.
Without loss of generality, let the compressing disk denoted by $D_0$ compress $S'$ into $V$.

Let $V' = \mathrm{cl}(V - N(D))$.
Then $V'$ has compressible boundary and $V' \cup N(E)$ also has compressible boundary because of the compressing disk $D_0$.
By Jaco's handle addition theorem \cite{Jaco},
there exists a compressing disk $D'$ in $V'$ which is disjoint from $\partial E$.
Then $(D,D',E)$ is the desired weak reducing triple such that $(D',E)$ is a weak reducing pair for $\partial V'$.
\end{proof}

\section{Adding two disjoint $2$-handles}\label{sec3}

Now we consider a slightly different situation.
Let $(D_1,E_1)$ and $(D_2,E_2)$ be two weak reducing pairs for a Heegaard splitting $V \cup_S W$
and $D_1 \cap D_2 = \emptyset$.
When is a weak reducing triple $(D_1,D_2,E)$ possible for some disk $E \subset W$?
A natural condition would be that $V$ compressed along $D_1 \cup D_2$ results a handlebody
whose boundary is compressible in the exterior.
An equivalent statement is that if we add two disjoint $2$-handles $N(D_1)$ and $N(D_2)$ to $W$,
then the resulting manifold $W'$ is boundary-reducible.
But the interior of a boundary-reducing disk for $W'$ may intersect $N(D_1)$ and $N(D_2)$.
Using Jaco's handle addition theorem and a generalization of it by Wu \cite{Wu},
we show that there exists a compressing disk for $W$ that is disjoint from $D_1$ and $D_2$.

\subsection{A generalization of the handle addition theorem}\label{subsec3.1}

Let $F$ be a surface in the boundary of a $3$-manifold $M$, and let $\gamma$ be a simple closed curve in $F$.
We call a compressing disk $D$ for $F$ an {\em $n$-compressing disk} (with respect to $\gamma$)
if $\partial D$ intersects $\gamma$ in $n$ points.
The surface $F$ is called {\em $n$-compressible} if an $n$-compressing disk exists.

Given a simple closed curve $J$ in $F$,
let $M'= \tau (M\,;\,J)$ be the manifold obtained by attaching a $2$-handle to $M$ along $J$.
Let $F' = \sigma (F\,;\,J)$ be the resulting surface obtained from $F$ after the $2$-handle addition.

\begin{theorem}[Wu]\label{thm_Wu}
Let $\gamma$ and $J$ be disjoint simple closed curves in $F$.
Suppose that $F - \gamma$ is compressible.
If $F'$ is $n$-compressible, then $F-J$ is $k$-compressible for some $k \le n$.
\end{theorem}

Note that if $\gamma = \emptyset$ and $n=0$, then it is the original Jaco's handle addition theorem.

By above result we can show the following.

\begin{lemma}\label{lem3.2}
Let $(D_1,E_1)$ and $(D_2,E_2)$ be two weak reducing pairs for a Heegaard splitting $V \cup_S W$
and $D_1 \cap D_2 = \emptyset$.
Suppose the manifold $W'$ obtained by adding two disjoint $2$-handles $N(D_1)$ and $N(D_2)$ to $W$ is boundary-reducible.
Then there exists a compressing disk for $W$ that is disjoint from both $D_1$ and $D_2$.
\end{lemma}

\begin{proof}
Let $S'$ be the surface $\sigma (S\,;\,\partial D_1)$.

\begin{claim}
$S'$ is compressible in $\tau (W\,;\,\partial D_1)$.
\end{claim}

\begin{proof}[Proof of Claim]
Suppose that $\partial E_1$ is inessential in $S'$.
Then $\partial E_1$ bounds a disk $\Delta$ in $S'$, and $\Delta$ contains either one or two scars of $D_1$.
In any case, push the interior of $\Delta$ slightly into $V$ to get a disk $\Delta'$.
Then $\Delta' \cup E_1$ is a reducing sphere for $V \cup_S W$.
So we can find an essential disk disjoint from $D_1$ and $\Delta' \cup E_1$ in a lower genus Heegaard splitting,
which is a compressing disk for $S'$ in $\tau (W\,;\,\partial D_1)$

If $\partial E_1$ is essential in $S'$,
then $S'$ is obviously compressible because $E_1$ is a compressing disk.
\end{proof}

Let $S''$ be the surface $\sigma (S'\,;\,\partial D_2)$.
Since $S'' = \partial W'$, by hypothesis $S''$ is compressible.
Hence by Jaco's handle addition theorem, $S' - \partial D_2$ is compressible.
It means that $S'$ is $0$-compressible with respect to $\partial D_2$.
Also note that $S - \partial D_2$ is compressible in $W$ because $E_2$ is a compressing disk.

Summarizing, we have the following.

$\bullet$ $S-\partial D_2$ is compressible in $W$.

$\bullet$ The surface $S'=\sigma(S\,;\,\partial D_1)$ is $0$-compressible with respect to $\partial D_2$.

By Theorem \ref{thm_Wu}, $S - \partial D_1$ is $0$-compressible with respect to $\partial D_2$.
Then there exists a compressing disk $E$ in $W$ which is disjoint from $\partial D_1$ and $\partial D_2$.
\end{proof}

\begin{theorem}\label{thm3.3}
Let $M$ be a boundary-reducible manifold and $\partial M$ be a genus $g$ surface.
If $M$ admits a genus $g+1$ Heegaard surface $S$, then the disk complex of $S$ is simply connected.
\end{theorem}

\begin{proof}
Let $V \cup_S W$ be a genus $g+1$ Heegaard splitting of $M$ with $\partial_- V = \partial M$.
By \cite[Theorem 2.5]{Bachman}, $\pi_1(\mathcal{D}(S)) = 1$ is equivalent to that $S$ is not critical.
For any two weak reducing pairs for $S$, it suffices to show that
there is a sequence of weak reducing pairs connecting the two weak reducing pairs.
Then there cannot be a partition $C_0 \cup C_1$ of the disk complex satisfying the criticality
because all weak reducing pairs belongs to only one set of the partition $C_0 \cup C_1$.

Since $V$ is obtained from $\partial_-V \times I$ by attaching one $1$-handle,
the structure of the set of compressing disks of $V$ is rather simple.
There is a unique nonseparating disk $D_1$ of $V$ and any separating essential disk $D_2$ of $V$ is disjoint from $D_1$.
So what remains to prove is that any two weak reducing pairs $(D_1,E_1)$ and $(D_2,E_2)$ for $S$
can be connected by a sequence of weak reducing pairs.
We can see that cutting $V$ along $D_1 \cup D_2$ results in $\partial M \times I$ and a $3$-ball,
whose exterior is a boundary-reducible manifold.
By Lemma \ref{lem3.2} there exists a compressing disk $E$ for $W$ that is disjoint from both $D_1$ and $D_2$.
See Figure \ref{fig2}.

\begin{figure}[ht!]
\begin{center}
\includegraphics[width=5.5cm]{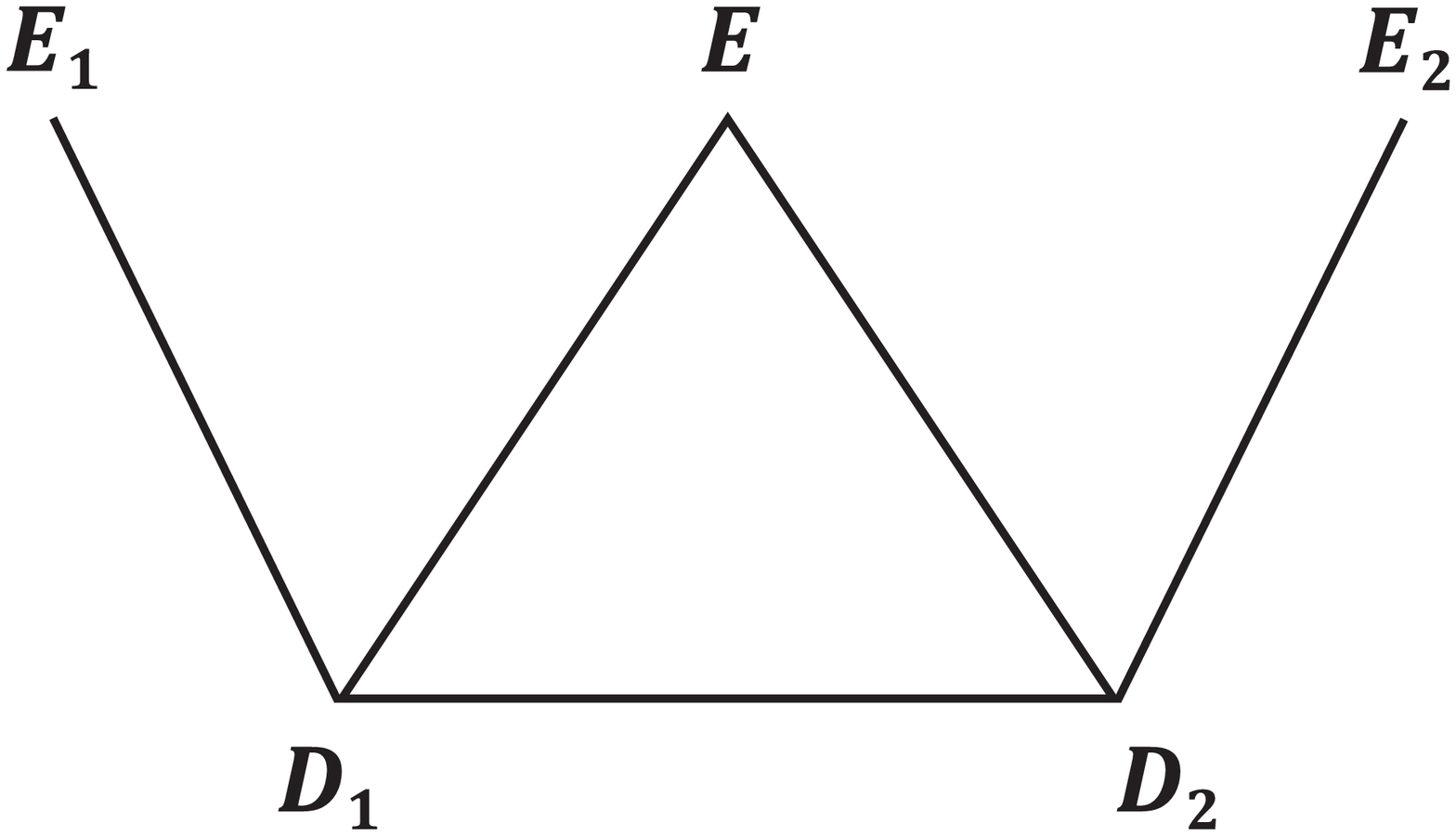}
\caption{}\label{fig2}
\end{center}
\end{figure}

\end{proof}

\section{Connectedness of the complex of reducing spheres}\label{sec4}

Let $S$ be a Heegaard surface of $M$.
A {\em reducing sphere} for $S$ is a sphere that intersects $S$ in a single essential loop.
The complex of reducing spheres for $S$ is defined in a similar manner as a disk complex.
For a genus two Heegaard splitting of $S^3$, the complex of reducing spheres is connected \cite{Scharlemann}.
In genus two case, the complex is connected in the sense that for any two reducing spheres $P$ and $Q$,
there is a sequence of reducing spheres $P_0, P_1, \ldots, P_n$ such that
$P_0 = P$ and $P_n =Q$ and $|P_i \cap P_{i+1}| = 4$ for $i = 0, \ldots, n-1$.
However, it is not known whether the complex of reducing spheres
for a genus $g \ge 3$ Heegaard splitting of $S^3$ is connected or not.

\begin{question}\label{quest4.1}
Is the complex of reducing spheres for a genus $g \ge 3$ Heegaard splitting of $S^3$ connected?
\end{question}

Let $M$ be a product manifold $\mathrm{(torus)} \times [0,1]$ and $V \cup_S W$ be a genus three Heegaard splitting of $M$
with $\partial_- V = \mathrm{(torus)} \times \{ 0 \}$ and $\partial_- W = \mathrm{(torus)} \times \{ 1 \}$.
The Heegaard surface $S$ can be obtained from $\mathrm{(torus)} \times \{ \frac{1}{2} \}$ by stabilizing twice.
So $V \cup_S W$ is a connected sum of a genus one trivial splitting of $M$ and a genus two Heegaard splitting of $S^3$.
The complex of reducing spheres for $S$ is larger than that of a genus two Heegaard surface of $S^3$ and
smaller than that of a genus three Heegaard surface of $S^3$.
Therefore it is worth to consider the complex of reducing spheres for $S$.

Specifically, let $P$ and $Q$ be non-isotopic reducing spheres for $S$ dividing each of $V$ and $W$ into
a $\mathrm{(torus)} \times I$ and a genus two handlebody.
Most of the properties that hold for $P$ and $Q$ also hold when the role of $P$ and $Q$ exchanged.
We investigate the intersection of $P$ and $Q$.
We assume that $P$ and $Q$ intersect transversely and minimally.
The reducing sphere $P$ divides $V$ into a genus two handlebody $V_1$ and a $\mathrm{(torus)} \times I$ component $V_2$.
Similarly the reducing sphere $Q$ divides $V$ into a genus two handlebody $\overline{V_1}$ and
a $\mathrm{(torus)} \times I$ component $\overline{V_2}$.
Let $D = Q \cap V$.
Since $|P \cap Q|$ is minimal, the intersection of two reducing disks $(P \cap V) \cap D$ consists of arcs.
The reducing disk $P \cap V$ is cut into subdisks by $D$, and also $D$ is cut into subdisks by $P \cap V$.
Each such subdisk, say a subdisk $D_i$ of $D$, is a $2n$-gon for some $n$,
where a subarc of $\partial D$ and an arc of $D \cap (P \cap V)$ appears alternately on $\partial D_i$.
Note that a bigon is an outermost disk.

\begin{lemma}\label{lem4.2}
Any outermost disk of $D$ cut by $P \cap V$ is contained in $V_1$.
\end{lemma}

\begin{proof}
Suppose an outermost disk $\Delta$ is contained in $V_2$.
Since $V_2$ is $\mathrm{(torus)} \times I$, $\Delta$ is an inessential disk in $V_2$.
So we can isotope $\Delta$ into $V_1$ and reduce $|P \cap Q|$, a contradiction.
\end{proof}

The disk $P \cap V$ cuts $\overline{V_1}$ and $\overline{V_2}$ into submanifolds.
Since every properly embedded disk in $\mathrm{(torus)} \times I$ is boundary-parallel,
$\overline{V_2}$ is cut into $3$-ball components and a $\mathrm{(torus)} \times I$ component.
Consider an outermost $3$-ball $B_o$ in $\overline{V_2}$ cut by an outermost disk.
The boundary of $B_o$ consists of three parts. See Figure \ref{fig3} for an example.

\begin{figure}[ht!]
\begin{center}
\includegraphics[width=5cm]{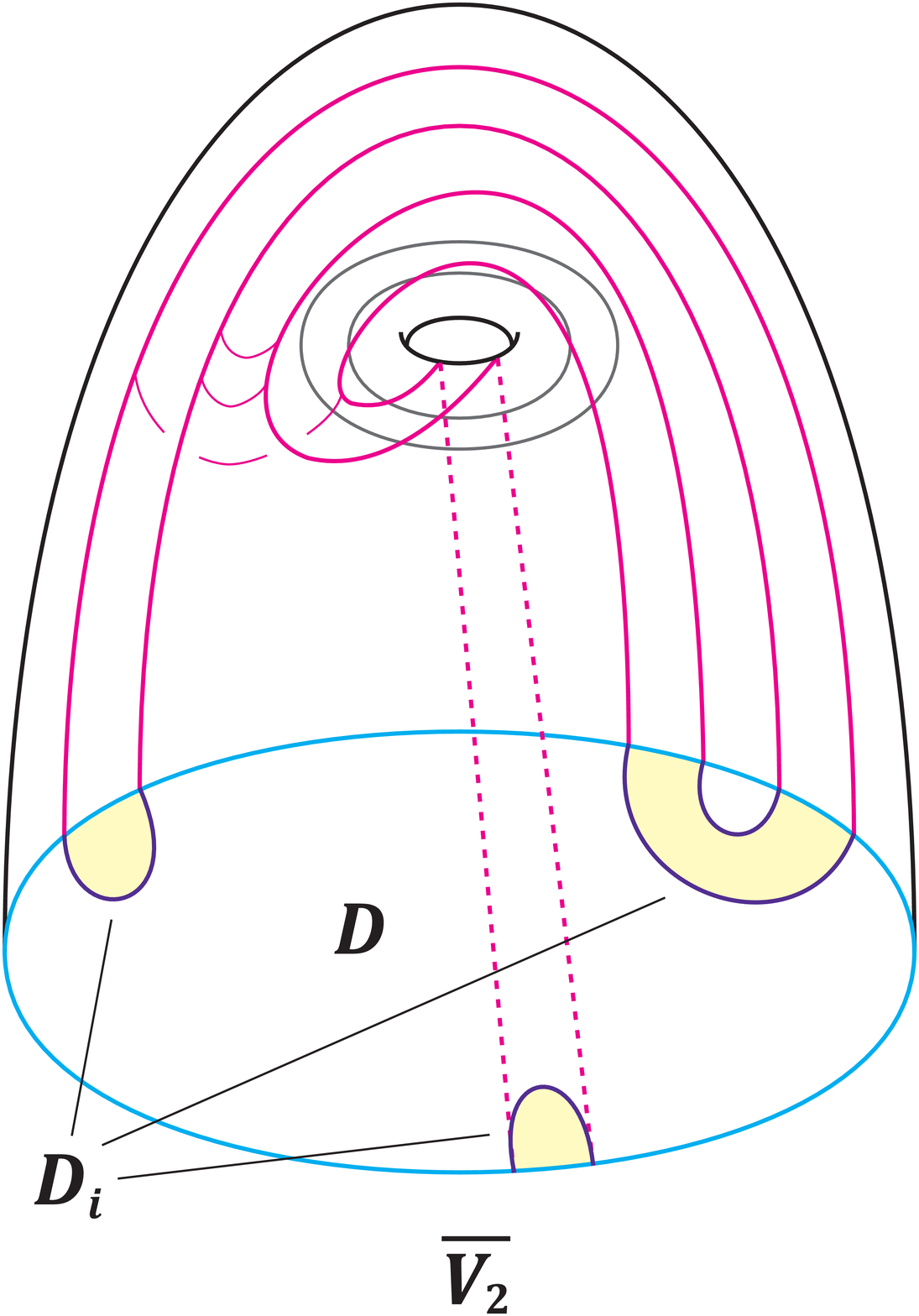}
\caption{An outermost $3$-ball $B_o$ in $\overline{V_2}$}\label{fig3}
\end{center}
\end{figure}

\begin{itemize}
\item a single subdisk $P_0$ of $P \cap V$ cut by $D$.
\item subdisk components $D_i$ of $D$($=Q \cap V$) cut by $P \cap V$.
\item $\partial B_o \cap S$ (possibly disconnected).
\end{itemize}

In other words, $\partial B_o = P_0 \cup (\bigcup D_i) \cup (\partial B_o \cap S)$.
Since $\partial B_o - P_0$ is a disk and
there can be no bigon component in $\partial B_o \cap S$ by minimality of $|P \cap Q|$,
there exist at least two bigons among the $D_i$'s.
By Lemma \ref{lem4.2}, each such bigon is contained in $V_1$.

For $W$, the same properties as above hold.
Let $P$ divide $W$ into a genus two handlebody $W_1$ and a $\mathrm{(torus)} \times I$ component $W_2$.
Let $Q$ divide $W$ into a genus two handlebody $\overline{W_1}$ and
a $\mathrm{(torus)} \times I$ component $\overline{W_2}$.
Let $E = Q \cap W$.
Let $\widetilde{B_o}$ be an outermost $3$-ball in $\overline{W_2}$ cut by $P \cap W$.
Then $\partial\widetilde{B_o}$ consists of three parts as above.

\begin{lemma}\label{lem4.3}
Suppose that a collection of essential disks $\mathcal{D}= \{ \Delta_i \}$ containing a non-separating disk $\Delta$
cuts a $3$-ball component $L$ from a genus two handlebody and $\partial L$ contains only one scar of $\Delta$.
Then one of the following holds. See Figure \ref{fig4}.

\begin{enumerate}
\item There is a disk $\Delta_1$ parallel to $\Delta$, and $\mathcal{D} = \{ \Delta, \Delta_1 \}$.
\item There are a disk $\Delta_1$ parallel to $\Delta$ and a separating disk $\Delta_2$, and
$\mathcal{D} = \{ \Delta, \Delta_1, \Delta_2 \}$.
\item There are a disk $\Delta_1$ parallel to $\Delta$ and two parallel disk $\Delta_2$ and $\Delta_3$, and
$\mathcal{D} = \{ \Delta, \Delta_1, \Delta_2, \Delta_3 \}$.
\item $\mathcal{D}$ is a collection of three mutually non-parallel disks $\{ \Delta, \Delta_2, \Delta_3 \}$.
\end{enumerate}
\end{lemma}

\begin{figure}[ht!]
\begin{center}
\includegraphics[width=9cm]{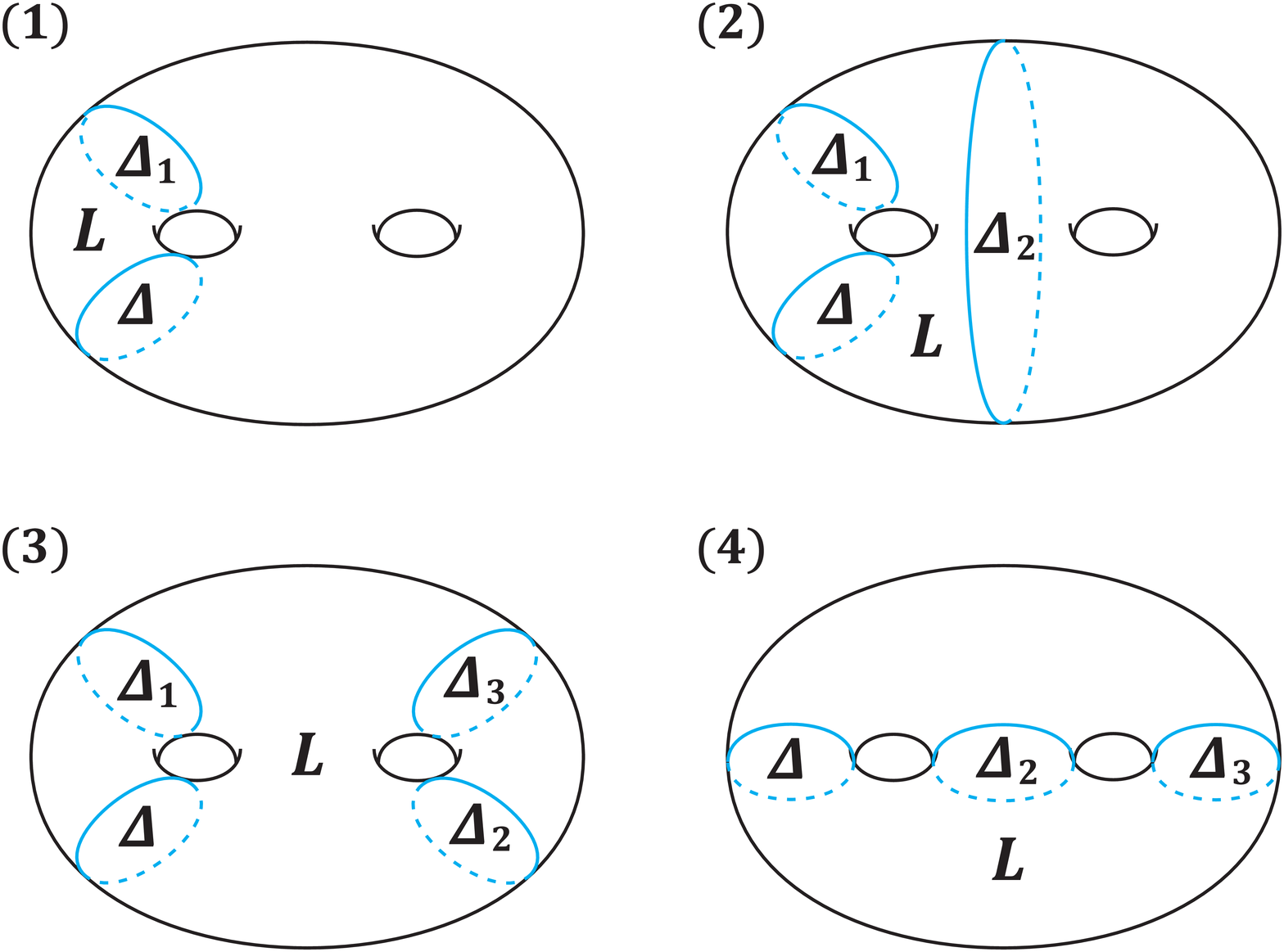}
\caption{}\label{fig4}
\end{center}
\end{figure}

\begin{proof}
The non-separating disk $\Delta$ cuts the genus two handlebody into a solid torus $T$ and
there are two scars of $\Delta$ on $\partial T$.
Suppose first that there is another disk of $\mathcal{D}$ which is inessential in $T$.
Because $\partial L$ contains only one scar of $\Delta$, there is a disk $\Delta_1$ that separates the two scars of $\Delta$.
It means that the disk in $\partial T$ bounded by $\partial\Delta_1$ contains only one scar of $\Delta$.
If $L$ is the $3$-ball that $\Delta_1$ cuts from $T$, then it is Case $(1)$.
If $L$ is in the complement of the $3$-ball that $\Delta_1$ cuts from $T$, then it is Case $(2)$ or $(3)$.

If every disk of $\mathcal{D}$ except $\Delta$ is essential in $T$, then it is Case $(4)$.
\end{proof}


\begin{definition}\label{def4.4}
A pair $(\Delta, \widetilde{\Delta})$ of an outermost disk of $D$ and $E$ respectively {\em cross} if
endpoints of $\partial\Delta \cap S$ and $\partial\widetilde{\Delta} \cap S$ appear alternately on $P \cap S$.
\end{definition}

\begin{proposition}\label{prop4.5}
For $P$ and $Q$, one of the following holds.

\begin{itemize}
\item There exist two outermost disks of $D$ (or $E$) which are parallel in $V_1$ (or $W_1$, respectively).
\item There exist a separating outermost disk of $D$ in $V_1$ and a separating outermost disk of $E$ in $W_1$,
      hence the subarcs of their boundaries in $S$ are parallel.
\item There exists a pair $(\Delta, \widetilde{\Delta})$ of an outermost disk of $D$ and $E$ respectively that cross.
\end{itemize}
\end{proposition}

\begin{proof}\label{proof_prop}
Let $B_o$ be an outermost $3$-ball in $\overline{V_2}$ cut by $P \cap V$,
and let $\Delta_i, \Delta_j$($\subset \partial B_o$) be two disjoint bigon subdisks of $D$ cut by $P \cap V$.
Note that $B_o$ is contained in $V_1$.
If both $\Delta_i$ and $\Delta_j$ are separating in $V_1$,
then $\Delta_i$ and $\Delta_j$ are parallel, so we get the first conclusion.
So we may assume that at least one, say $\Delta = \Delta_i$, is non-separating in $V_1$.

Suppose that there is any subdisk component $D_i$ of $D$ in $\partial B_o$ which is inessential in $V_1$.
Then $D_i$ cuts a $3$-ball $B_i$ from $V_1$.
We add such $B_i$'s to $B_o$ for all $i$, and let $B$ be the resulting $3$-ball.
We can see that $P_1 = \partial B \cap (P \cap V)$ is a single disk.
Since $\Delta$ is a subdisk of a separating reducing disk, $\partial B$ contains only one scar of $\Delta$.

Similarly, let $\widetilde{B_o}$ be an outermost $3$-ball in $\overline{W_2}$ cut by $P \cap W$
and $\widetilde{B}$ be obtained by adding $3$-balls to $\widetilde{B_o}$ so that
any subdisk component of $E$ in $\partial\widetilde{B}$ is essential in $W_1$.
Let $\widetilde{\Delta} \subset \partial\widetilde{B}$ be a bigon subdisk of $E$ that is non-separating in $W_1$.
By Lemma \ref{lem4.3}, $\Delta$ and $\widetilde{\Delta}$ belong to one of the four cases.

\vspace{0.3cm}

Case $1$.\, There is a disk $\Delta_1$ parallel to $\Delta$, and $\mathcal{D} = \{ \Delta, \Delta_1 \}$.

Both $\Delta$ and $\Delta_1$ are bigons, and we get the first conclusion.

\vspace{0.3cm}

Case $2$.\, There are a disk $\Delta_1$ parallel to $\Delta$ and a separating disk $\Delta_2$, and
$\mathcal{D} = \{ \Delta, \Delta_1, \Delta_2 \}$.

The separating disk $\Delta_2$ cuts $V_1$ into two solid tori $T_1$ and $T_2$, and suppose that $\Delta$ is in $T_1$.
Let $B'$ be the $3$-ball of the parallelism between $\Delta$ and $\Delta_1$.
Since there are at least two bigons among the $\Delta_i$'s, $\Delta_1$ or $\Delta_2$ is a bigon.
If $\Delta_1$ is a bigon, then we get the first conclusion.
So we may assume that $\Delta_2$ is a bigon.

Consider the non-separating outermost disk $\widetilde{\Delta}$ of $E$ cut by $P \cap W$ mentioned above.
Let $\alpha = \partial\widetilde{\Delta} \cap S$ and let $a$ and $b$ be the two endpoints of $\alpha$.
Let $\beta$ be a subarc of $P \cap S$ with $\partial\beta = \{ a,b \}$, anyone among the two choices.
Then $\alpha \cup \beta$ bounds a non-separating disk in $W_1$.

\vspace{0.2cm}

Case $2.1$.\, $\alpha \subset B \cap S$

Since $a,b \in \partial P_1$ and $P_1 \subset P \cap V$,
the arc $\beta$ is isotopic in $P \cap V$ to a subarc of $\partial P_1$.
Hence $\alpha \cup \beta$ is isotopic to a loop in $\partial B$.
So $\alpha \cup \beta$ bounds a disk also in $V_1$, a contradiction.

\vspace{0.2cm}

Case $2.2$.\, $\alpha \subset B' \cap S$

Let $A$ be the annulus $B' \cap \partial T_1$.
If $A \cap (P \cap V)$ has a $2n$-gon component for some $n \ge 2$,
then $\partial B \cap (P \cap V)$($ = P_1$) would be disconnected, a contradiction.
Hence $A \cap (P \cap V)$ consists of bigon disks.
Then $A' = A \cap S$ is also an annulus.

\begin{claim}
If both endpoints of $\alpha$ are on the same component of $\partial A'$,
then $\alpha \cup \beta$ bounds a disk in $V_1$, a contradiction.
\end{claim}

\begin{proof}[Proof of Claim]
The arc $\alpha$ is isotopic in $B'$ to an arc $\gamma$ in $\Delta$ or $\Delta_1$, say $\Delta_1$.
The arc $\beta$ is isotopic in $P \cap V$ to a subarc $\delta$ of $\partial P_1$, with $\partial\delta = \partial\gamma$.
Since $\gamma \cup \delta$ is in $\partial B$, $\gamma \cup \delta$ bounds a disk in $V_1$.
Therefore $\alpha \cup \beta$ bounds a disk in $V_1$, a contradiction.
\end{proof}

If $a$ and $b$ are on different components of $\partial A'$, then $\Delta$ and $\widetilde{\Delta}$ cross.

\vspace{0.2cm}

Case $2.3$.\, $\alpha \subset T_2$.

Now we consider the four cases for $W_1$.
The disk $\widetilde{\Delta}$ corresponds to one of the cases in Case $1$ -- Case $4$ with the relevant notations.
In Case $1$, we get the first conclusion.
In Case $2$, we get the first conclusion, or there is a separating bigon disk in $W_1$ as we observed,
hence the second conclusion holds.
In Case $3$ and $4$, we get conclusions as will be explained below.

\vspace{0.3cm}

Case $3$.\, There are a disk $\Delta_1$ parallel to $\Delta$ and two parallel disk $\Delta_2$ and $\Delta_3$, and
$\mathcal{D} = \{ \Delta, \Delta_1, \Delta_2, \Delta_3 \}$

Remember that on $\partial B$ there are at least two bigon subdisks of $D$.
If $\Delta_1$ is a bigon disk, then $\Delta$ and $\Delta_1$ are parallel in $V_1$ and we get the first conclusion.
So without loss of generality, assume that $\Delta$ and $\Delta_2$ are bigon disks.
Consider the non-separating outermost disk $\widetilde{\Delta}$ of $E$ cut by $P \cap W$ and
the arcs $\alpha$ and $\beta$ as in Case $2$.
Then $\alpha \cup \beta$ bounds a non-separating disk in $W_1$.
Let $B'_1$ and $B'_2$ the two $3$-balls $V_1 - B$.

If $\alpha \subset B \cap S$, then $\alpha \cup \beta$ bounds a disk in $V_1$, a contradiction.
Suppose $\alpha \subset B'_1 \cap S$ or $\alpha \subset B'_2 \cap S$, say $\alpha \subset B'_1 \cap S$.
Let $A$ be the annulus $B'_1 \cap \partial V_1$.
Then $A' = A \cap S$ is also an annulus.
If both endpoints of $\alpha$ are on the same component of $\partial A'$,
then $\alpha \cup \beta$ bounds a disk in $V_1$, a contradiction.
If two endpoints of $\alpha$ are on different components of $\partial A'$,
then $\Delta$ and $\widetilde{\Delta}$ cross.

\vspace{0.3cm}

Case $4$.\, $\mathcal{D}$ is a collection of three mutually non-parallel disks $\{ \Delta, \Delta_2, \Delta_3 \}$.

Suppose that $\Delta$ and $\Delta_2$ are bigon disks without loss of generality.
Let $\widetilde{\Delta}$, $\alpha$ and $\beta$ be as above.
Then $\alpha \cup \beta$ bounds a non-separating disk in $W_1$.
Let $B'$ be the $3$-ball $V_1 - B$.

If $\alpha \subset B \cap S$, then $\alpha \cup \beta$ bounds a disk in $V_1$, a contradiction.
Suppose $\alpha \subset B' \cap S$.
Let $R$ be the pair of pants $B' \cap \partial V_1$.
Then $R' = R \cap S$ is also a pair of pants.
If both endpoints of $\alpha$ are on the same component of $\partial R'$,
then $\alpha \cup \beta$ bounds a disk in $V_1$, a contradiction.
If two endpoints of $\alpha$ are on different components of $\partial R'$,
then $\Delta$ and $\widetilde{\Delta}$ cross.
\end{proof}

If $\Delta$ and $\widetilde{\Delta}$ cross, then the loop obtained from $P \cap S$
by banding along the arcs $\partial\Delta \cap S$ and $\partial\widetilde{\Delta} \cap S$ bounds disks
in both $V_1$ and $W_1$.
So it results a reducing sphere $R$ with $|P \cap R| = 4$ and $|R \cap Q| < |P \cap Q|$.
However, $R$ does not divide $V$ into $\mathrm{(torus)} \times I$ and a genus two handlebody.
(It divides $V$ into a solid torus and a compression body.)
An investigation of intersection of reducing spheres such as $R \cap Q$
is left as further study.


\end{document}